\documentclass[12pt]{article}
\usepackage{latexsym,amsmath,amsthm,verbatim,ifthen,amssymb}
\usepackage{color}
\usepackage[all]{xy}

\newtheorem{theorem}{Theorem}
\newtheorem{corollary}[theorem]{Corollary}
\newtheorem{lemma}[theorem]{Lemma}
\newtheorem{proposition}[theorem]{Proposition}

\newcommand{\mQ}{\mathbb{Q}}

\newcommand{\pa}{\left(}
\newcommand{\pb}{\right)}
\newcommand{\ca}{\left\{}
\newcommand{\cb}{\right\}}

\newcommand{\quism}{\stackrel{\simeq}{\longrightarrow}}

\newcommand{\secat}{{\rm secat}}

\newcommand{\msecat}{{\rm msecat}}

\newcommand{\cat}{{\rm cat}}

\newcommand{\tc}{{\rm TC}}
\newcommand{\mtc}{{\rm mTC}}

\newcommand{\mcat}{{\rm mcat}}

\newcommand{\id}{{\rm Id}}

\newcommand{\cdga}{{\bf cdga}}

\newcommand{\apl}{ {A_{PL}}}

\begin{document}
\title{Module sectional category of products}
\author{J.G. Carrasquel-Vera\footnote{Supported by the Belgian Interuniversity Attraction Pole (IAP) within the framework ``Dynamics, Geometry and Statistical Physics'' (DYGEST P7/18)}, P.-E. Parent, L. Vandembroucq}

\maketitle
\begin{abstract}
Extending a result of F\'elix-Halperin-Lemaire on Lusternik-Schnirelmann category of products, we prove additivity of a rational approximation for Schwarz's sectional category with respect to products of fibrations. \\
\end{abstract}

 \vspace{0.5cm}
 \vspace{0.2cm}

\section*{Introduction}
The sectional category\cite{Schwarz66} (or Schwarz genus) of a fibration $p\colon E\rightarrow X$ is the smallest integer $n$ such that $X$ admits a cover by open sets on each of which a local section for $p$ exits. This homotopy invariant is a generalization of the well known Lusternik-Schnirelmann (LS) category\cite{Lusternik34} of a path-connected space $X$, $\cat(X)$, as it is the sectional category of the path fibration $PX\rightarrow X$, $\alpha\mapsto \alpha(1)$, where $PX$ is the space of paths starting at the base point.\\

One of the most important results of \cite{FHL98} says that, if $X$ and $Y$ are simply connected rational spaces of finite type, then $\cat(X\times Y)=\cat(X)+\cat(Y)$. This was done through Hess' theorem \cite{He91} by proving the analogous result for a lower bound of LS category called module LS category.\\

Throughout this paper we will consider all spaces to be simply connected CW-complexes of finite type. We will also denote $f_0$ the rationalisation of a map $f$. As for LS category, there exists a lower bound of sectional category, called module sectional category \cite{FGKV06}, for which we have $\mcat(X)=\msecat(PX\rightarrow X)$. In this paper we prove

\begin{theorem}\label{th:Main}
Let $f$ and $g$ be two fibrations. If $f_0$ admits a homotopy retraction, then \[\msecat(f\times g)=\msecat(f)+\msecat(g).\]
\end{theorem}

Another important particular case of sectional category is Farber's (higher) topological complexity\cite{Farber03,Rudyak10} of a space $X$, $\tc_n(X)=\secat(\pi_n)$, where the fibration $\pi_n\colon X^{[1,n]}\rightarrow X^n$ is such that $\pi_n(\alpha)=\pa\alpha(1),\alpha(2),\ldots,\alpha(n)\pb$.\\

As a direct application of Theorem \ref{th:Main}, the module invariant associated to (higher) topological complexity, \[\mtc_n(X):=\msecat(\pi_n),\] 
is additive:

\begin{corollary}
Let $X$ and $Y$ be two spaces. Then \[\mtc_n(X\times Y)=\mtc_n(X)+\mtc_n(Y).\]
\end{corollary}

The results given are an improvement of \cite{Carrasquel15a}.
 
\section{Preliminaries}
This section contains a brief summary of the tools that will be used, see \cite{Bible} for further details. Let $(A,d)$ be a commutative differential graded algebra  over $\mQ$ (\cdga). An $(A,d)$-module is a chain complex $(M,d)$ together with a degree $0$ action of $A$ verifying that $d(ax)=d(a)x+(-1)^{\deg(a)}ad(x)$. The module $M^{\#}=\hom(M,\mQ)$ admits an $(A,d)$-module structure with action $(a\varphi)(x)=(-1)^{\deg(a)\deg(\varphi)}\varphi(ax)$ and differential $d(\varphi)=(-1)^{\deg(\varphi)}\varphi\circ d$. If $N$ is an $(A,d)$-modules, then the module $M\otimes_AN$ admits an $(A,d)$-module structure with action $a(m\otimes n)=(am)\otimes n$ and differential $d(m\otimes n)=d(m)\otimes n+(-1)^{\deg(m)}m\otimes d(n)$. An $(A,d)$-module $P$ is said to be semifree if there exists an increasing filtration $P_*$ by $(A,d)$-submodules such that $P_k/P_{k-1}$ is a free $(A,d)$ module on a basis of cocycles. Every $(A,d)$-module $M$ admits a semifree resolution, that is a quasi-isomorphism of $(A,d)$ modules, $P\quism M$, where $P$ is $(A,d)$-semifree. If $P$ is $(A,d)$-semifree and if $\eta$ is a quasi-isomorphism of $(A,d)$-modules then $\eta\otimes_A \id_P$ and $\id_P\otimes_A \eta$ are also quasi-isomorphisms. A morphism of $(A,d)$-modules $\varphi\colon (M,d)\rightarrow (N,d)$ is said to have a homotopy retraction if there exists a commutative diagram of $(A,d)$-modules,
\[\xymatrix{
(M,d)\ar[r]^{\id}\ar[d]_\varphi\ar[dr] &(M,d)\\
(N,d)&(P,d).\ar[l]^-\simeq\ar[u]\\
}\]

We will use the following lemma which is an expression of one of the central ideas of \cite{FHL98}.

\begin{lemma}\label{lem:retractHinj}
Let $\varphi\colon (A,d)\rightarrow (B,d)$ be a surjective \cdga\ morphism with kernel $K$ and $A$ of finite type. The morphism $\varphi$ admits a homotopy retraction of $(A,d)$-modules if and only if for any $(A,d)$ semi-free resolution $\eta\colon P\quism A^\#$, the projection \[\varrho\colon P\longrightarrow \frac{P}{K\cdot P}\] is injective in homology.
\end{lemma}

\begin{proof}
Suppose that $\varphi$ admits a homotopy retraction of $(A,d)$-module.  This means that there exists a commutative diagram of $(A,d)$ module of the form
$$\xymatrix{
A\ar[r]^{\id_A} \ar[d]_{\varphi}_{} \ar[dr]^{i}_{}
& A
\\ 
B
& M \ar[l]^-{\simeq} \ar[u]^{}_{r} 
}$$
where we can suppose that $M$ is a $(A,d)$ semi-free resolution. 
Let now $P\quism A^{\#}$ be a $(A,d)$ semi-free resolution and apply $-\otimes_A P$ to the diagram above. We get
 $$\xymatrix{
P\ar[r]^{\id_P} \ar[d]_{}_{} \ar[dr]^{}_{}
& P
\\ 
B\otimes_A P
& M\otimes_A P. \ar[l]^-{\simeq} \ar[u]^{}_{} 
}$$

Since $B=\frac{A}{K}$ we have $B\otimes_A P=\frac{P}{K\cdot P}$ and the left hand morphism is the projection $\varrho\colon P\to \frac{P}{K\cdot P}$. The diagram shows that $\varrho$ admits a homotopy retraction of $(A,d)$-module and therefore that it is injective in homology.\\

Conversely, suppose that $\varrho$ is homology injective. Since $A$ is of finite type, $\eta^\#=\hom(\eta,\mQ)\colon A\rightarrow \hom(P,\mQ)$ is also a semi-free resolution. Since $\varrho$ is homology injective, \[\varrho^\#=\hom(\varrho,\mQ)\colon \hom\pa \frac{P}{K\cdot P}, \mQ\pb\rightarrow \hom(P,\mQ)\] is homology surjective. There exists then a cycle $f\in \hom\pa \frac{P}{K\cdot P},\mQ\pb$ such that $[f\circ \varrho]=[z]$, where $z=\eta^\#(1)$. Now define the $(A,d)$-module morphism $\alpha\colon A\rightarrow \hom(\frac{P}{K\cdot P},\mQ)$ as $\alpha(1)=f$. Then $\varrho^\#\circ \alpha$ is homotopic to $\eta^\#$ and thus a quasi-isomorphism. To finish the proof, we observe that $K\cdot \hom\pa\frac{P}{K\cdot P},\mQ\pb=\ca 0\cb$ and thus we have a diagram
\[\xymatrix{
A\ar[rr]^{\id_A}\ar[d]_\varphi&&A\ar[d]^{\varrho^\#\circ \alpha}_\simeq\\
B\ar[r]_-{\overline{\alpha}} &\hom(\frac{P}{K\cdot P},\mQ)\ar[r]_{\varrho^\#} &\hom(P,\mQ),\\
}\]
which yields a homotopy retraction for $\varphi$ as $(A,d)$-modules.
\end{proof}

Let us denote by $p_n\colon J^n_X(E)\to X$ the join of $n+1$ copies of a fibration $p\colon E\rightarrow X$. As it is well-known \cite{Schwarz66}, $\secat(p)\le n$ if and only if $p_n$ admits a homotopy section. By definition, $\msecat(p)$ is the smallest $n$ such that  $\apl(p_n)$ admits a homotopy retraction of $\apl(X)$-modules, where $\apl$ denotes Sullivan's functor of piecewise linear forms\cite{Su77}.\\

Recall the following general characterization of $\msecat(f)$ from \cite{FGKV06}. Let $(A,d)\to (A\otimes (\mQ\oplus X),d)$ be a semi-free extension of $(A,d)$-module which is a model for $f$. For $x\in X$, write $dx=d_0x+d_+x\in A\oplus A\otimes X$. Then $\msecat(f)$ is the least $m$ such that the following $(A,d)$ semi-free extension admits a retraction of $(A,d)$-module: 
$$j_m\colon(A,d) \to J_m=(A \otimes  (\mQ\oplus s^{-m}X^{\otimes m+1},d).$$
Here $d=d_0+d_+$ (in $A \oplus A\otimes s^{-m}X^{\otimes m+1}$) is given by 
\begin{eqnarray*}
\lefteqn{d(s^{-m}x_0\otimes \cdots \otimes x_m) = (-1)^{\sum \limits_{k=1}^{m}(k|x_{m-k}| + k -1)} d_0x_0\cdot \cdots \cdot d_0x_m}\\
&+ & \sum \limits _{i=0}^{m} \sum \limits_{j_i}(-1)^{(|a_{ij_i}| +1)(|x_0| + \cdots +|x_{i-1}| + m)}a_{ij_i}\otimes s^{-m}x_0 \otimes \cdots \otimes x_{ij_i}\otimes  \cdots \otimes x_m,
\end{eqnarray*}
for $x_0$,..., $x_m\in X$ and $d_+x_i = \sum \limits_{j_i}a_{ij_i}\otimes x_{ij_i}$ with $a_{ij_i}\in A$ and $x_{ij_i}\in X$.

Using the following notation (suggested by the standard rules of signs)
$$s^{-m}x_0 \otimes \cdots \otimes d_+x_{i}\otimes  \cdots \otimes x_m:=\sum \limits_{j_i}\sigma_{ij_i}a_{ij_i}\otimes s^{-m}x_0 \otimes \cdots \otimes x_{ij_i}\otimes  \cdots \otimes x_m$$
we can write $d_+(s^{-m}x_0\otimes \cdots \otimes x_m)$ as
\begin{eqnarray*}
d_+(s^{-m}x_0\otimes \cdots \otimes x_m) = (-1)^m\sum \limits _{i=0}^{m} \sum \limits_{j_i}\tau_i s^{-m}x_0 \otimes \cdots \otimes d_+x_{i}\otimes  \cdots \otimes x_m,
\end{eqnarray*}
where $\sigma_{ij_i}:=(-1)^{|a_{ij_i}|(|x_0| + \cdots +|x_{i-1}| + m)}$ and $\tau_i:=(-1)^{(|x_0| + \cdots +|x_{i-1}|)}$.\\

Now let $f$ be a fibration such that $f_0$ admits a homotopy retraction. Then by \cite{Carrasquel14c}, there exists a surjective model for $f$, $\varphi\colon A\rightarrow \frac{A}{K}$ (called s-model) such that $\msecat(f)$ is the smallest $m$ for which the projection $\rho_m\colon A\rightarrow \frac{A}{K^{m+1}}$ admits a homotopy retraction of $(A,d)$-modules. We have

\begin{proposition}\label{Prop:fRetract}
Let $f$ be a fibration such that $f_0$ admits a homotopy retraction, $\varphi\colon A\rightarrow \frac{A}{K}$ and s-model for $f$ and $(A,d)\to (A\otimes (\mQ\oplus X),d)$ a semifree model for $f$, as in previous paragraphs. Let also $\eta\colon P\quism A^\#$ be an $(A,d)$ semi-free resolution. Then the following are equivalent
\begin{itemize}
\item[(i)] $\msecat(f)\le m$,
\item[(ii)] the morphism $\id_P\otimes_Aj_m\colon P\rightarrow P\otimes(\mQ\oplus s^{-m}X^{\otimes m+1})$ is injective in homology, 
\item[(iii)] the projection $P\rightarrow \frac{P}{K^{m+1}\cdot P}$ is injective in homology.
\end{itemize}
\end{proposition}
\begin{proof}
By \cite{Carrasquel14c}, there is a diagram
$$\xymatrix{
& A \ar[ld]_{j_m}^{ h_A} \ar[d] \ar[rd] &\\
J_m \ar[r]_{\simeq} & C & \frac{A}{K^{m+1}}\ar[l]
}$$ 
where the left hand triangle is commutative up to a homotopy of $(A,d)$-modules and the right hand triangle is strictly commutative. Applying to previous diagram $\id_P\otimes_A-$, we get the following diagram of $(A,d)$ module:
$$\xymatrix{
& P \ar[ld]_{\id_P\otimes_Aj_m} \ar[d] \ar[rd] &\\
P\otimes (\mQ\oplus s^{-m}X^{\otimes m+1}) \ar[r]_-{\simeq} & P\otimes_A C & \frac{P}{K^{m+1}\cdot P}\ar[l]
}$$ 
where the left hand triangle is commutative up to a homotopy of $(A,d)$-module and the right hand triangle is strictly commutative. The result then follows from Lemma \ref{lem:retractHinj}.
\end{proof}

\section{The main result}
Observe that Proposition \ref{Prop:fRetract} together with the strategy of \cite{FHL98} can be used to easily prove Theorem \ref{th:Main} provided that both fibrations admit a homotopy retractions. In line with our statement, we here present a proof of the additivity of module sectional category when only one of the fibrations admits homotopy retraction.\\

We first notice that one of the inequalities of Theorem \ref{th:Main} follows in general:

\begin{proposition} Let $p\colon E\to X$ and $p'\colon E'\to X'$ be two fibrations. 
We have \[\msecat(p\times p')\leq \msecat(p) + \msecat(p').\]
\end{proposition}

\begin{proof}
In \cite[Pg. 26]{GGV15}, a commutative diagram of the following form is constructed:
$$\xymatrix{
J^n_X(E) \times J^m_{X'}(E')\ar[rr]^{\psi_{n,m}^{E,E'}} \ar[rd]_{p_n\times p'_m} && J^{m+n}_{X\times X'}(E\times E') \ar[dl]^{(p\times p')_{n+m}}\\
& X\times X'.
}$$
By applying $\apl$ to this diagram, we can establish that, if $\msecat(p)\leq n$ and $\msecat(p')\leq m$ then $\msecat(p\times p')\leq n+m$.
\end{proof}

Keeping the notation of Proposition \ref{Prop:fRetract}, the differential in $(A\otimes (\mQ\oplus X),d)$ can be taken such that $d_0(x)\in K$. This implies that, in $P\otimes (\mQ\oplus s^{-m}X^{\otimes m+1})$, $d_0(s^{-m}X^{\otimes m+1})\subset K^{m+1}\cdot P$. With this in mind we proceed to the 


\begin{proof}[Proof of Theorem \ref{th:Main}]
Take an s-model for $f$, $\varphi$ and an $(A,d)$ semi-free extension of $\varphi$, $A\otimes (\mQ\oplus X)$,
as in the previous section. Let also $(B,d)\to (B\otimes (\mQ\oplus Y),d)$ a semi-free model of $g$ (where $(B,d)$ is a \cdga\ model of the base of $g$). Then $f\times g$ is modeled by the tensor product of the two semi-free extensions which gives a semi-free extension of $(A\otimes B,d)$-modules that we write as follows
$$A\otimes B \to A\otimes B \otimes (\mQ\oplus Z) \qquad \mbox{where  } Z=X\oplus Y\oplus X\otimes Y.$$
In order to prove the statement, we suppose $\msecat(f)=m$ and $\msecat(f\times g)\leq m+p$ and we establish that $\msecat(g)\leq p$.\\

Let $P\quism A^{\#}$ be an $(A,d)$ semi-free resolution. Since $\msecat(f)=m$ we know from Proposition \ref{Prop:fRetract} that there exists $\Omega \in H(K^{m}\cdot P)$ which is not trivial in $H(P)$. Then there exist a cocyle $\omega \in K^{m}\cdot P$ representing $\Omega$ in $H(P)$ and $\theta \in P\otimes s^{-(m-1)}X^{\otimes m}$ such that $d\theta=\omega$. As a chain complex, we can write $P=\omega\cdot\mQ\oplus S$ where $d(S)\subset S$, and we define the following linear map of degree $-|\omega|$:
$$I_{\omega}\colon P\to \mQ, \quad I_{\omega}(\omega)=1, \,\,\, I_{\omega}(S)=0.$$
This map commutes with differentials. Now write the element $\theta \in P\otimes s^{-(m-1)}X^{\otimes m}$ as 
$$\theta=\sum \limits_{i}m_{i}\otimes s^{-(m-1)}x_i$$
 with $m_{i}\in P$ and $x_i\in X^{\otimes m}$. Since $d\theta=\omega$ we have $d_+\theta=0$ and $d_0\theta=\omega$.\\
 
Let $\psi\colon B\otimes (\mQ\oplus s^ {-p}Y^ {\otimes p+1}) \to P\otimes B \otimes (\mQ\oplus s^{-m-p}Z^{\otimes m+p+1})$ be the $B$-linear map of degree $|\omega|$ given by $\psi(1)=\omega \otimes 1$ and, for $y\in Y^{\otimes p+1}$,
$$
\begin{array}{l}
\psi(s^{-p}y)=-(-1)^{p|\omega|}\sum \limits_{i}(-1)^{(p+1)|m_i|}m_{i}\otimes 1\otimes s^{-m-p} x_{i}\otimes y
\end{array}$$
and extended to $ B\otimes (\mQ\oplus s^ {-p}Y^ {\otimes p+1})$ by the rule $\psi(b\cdot x)=(-1)^{|b||\omega|}b\cdot \psi(x)$. Notice that the structure of $(B,d)$-module on $P\otimes B \otimes (\mQ\oplus s^{-m-p}Z^{\otimes m+p+1})$ is given by $b\cdot (m\otimes b'\otimes z)=(-1)^{|m||b|}m\otimes b b'\otimes z$. In particular $\psi(b)=\omega\otimes b$. Let us now see that $\psi$ commutes with differentials, that is $\psi \circ d=(-1)^{|\omega|}d\circ \psi$. Since $\psi$ is $B$-linear and since $\omega$ is a cocycle we only have to see that
$$d\psi(s^{-p}y)=(-1)^{|\omega|}\psi(ds^{-p}y),$$ for each $y\in Y^{\otimes p+1}$.
Writing the differential of $P\otimes B \otimes (\mQ\oplus s^{-m-p}Z^{\otimes m+p+1})$ as $$d=d_0+d_+\in P\otimes B \oplus P\otimes B\otimes s^{-m-p}Z^{\otimes m+p+1}$$ we can check that
\begin{itemize}
\item $d_0\psi(s^{-p}y)=(-1)^ {|\omega|}\psi(d_0s^{-p}y)$ using the fact that $d_0\theta=\omega$, and
\item $d_+\psi(s^{-p}y)=(-1)^ {|\omega|}\psi(d_+s^{-p}y)$ using the fact
that $d_+\theta=0$. 
\end{itemize}

\noindent From $\msecat(f\times g)\leq m+p$ we know that the morphism
$$j_{m+p}^{A\otimes B}\colon A\otimes B \to A\otimes B \otimes (\mQ\oplus s^{-m-p}Z^{\otimes m+p+1})$$
admits a retraction $r$ of $(A\otimes B,d)$-modules. 
Finally the composite 

$$\xymatrix{
B\otimes (\mQ\oplus s^ {-p}Y^ {\otimes p+1})\ar[r]^-{\psi} & P\otimes B \otimes (\mQ\oplus s^{-m-p}Z^{\otimes m+p+1})\ar[d]^{P\otimes_A r}\\
&P\otimes B \ar[r]^{I_{\omega} \otimes \id} & B
}$$
gives a morphism (of degree $0$) of $(B,d)$-module which is a retraction for the inclusion $B\to B\otimes (\mQ\oplus s^ {-p}Y^ {\otimes p+1})$.
This proves that $\msecat(g)\leq p$.
\end{proof}

\section*{Acknowledgements}
The first author acknowledges the Belgian Interuniversity Attraction Pole (IAP) for support within the framework ``Dynamics, Geometry and Statistical Physics'' (DYGEST).

\vspace{2cm}
\noindent Institut de Recherche en Math\'ematique et Physique,\\
Universit\'e catholique de Louvain,\\
2 Chemin du Cyclotron,\\
1348 Louvain-la-Neuve, Belgium.\\
E-mail: \texttt{jose.carrasquel@uclouvain.be}

\vspace{1cm}

\noindent Centro de Matem\'atica,\\
Universidade do Minho,\\
Campus de Gualtar,\\
4710-057 Braga, Portugal.\\
E-mail: \texttt{lucile@math.uminho.pt}

\vspace{1cm}

\noindent Department of Mathematics and Statistics,\\
University of Ottawa,\\
585 King Edward Ave.,\\
Ottawa, Ontario, K1N 6N5, Canada.\\
E-mail: \texttt{pparent@uottawa.ca}

\end{document}